\documentclass[12pt]{amsart}

\usepackage{tikz}

\newcommand{\id}{\mathrm{id}}

\theoremstyle{plain}
\newtheorem{theorem}{Theorem}[section]
\newtheorem*{theorem*}{Main Theorem}
\newtheorem{lemma}[theorem]{Lemma}
\newtheorem{proposition}[theorem]{Proposition}
\newtheorem{corollary}[theorem]{Corollary}

\theoremstyle{definition}
\newtheorem{example}{Example}[section]

\DeclareMathOperator{\cyc}{cyc}
\DeclareMathOperator{\ecyc}{ecyc}
\DeclareMathOperator{\codim}{codim}
\DeclareMathOperator{\Inv}{Inv}
\DeclareMathOperator{\fl}{fl}

\title{Hultman elements for the hyperoctahedral groups}

\author{Alexander Woo}
\address{Dept.~of Mathematics, University of Idaho, Moscow, ID 83844-1103}
\email{awoo@uidaho.edu}
\date{\today}
\begin{document}

\begin{abstract}
Hultman, Linusson, Shareshian, and Sj\"ostrand gave a pattern avoidance characterization of the permutations
for which the number of chambers of its associated inversion arrangement is the same as the size of its lower interval
in Bruhat order.  Hultman later gave a characterization, valid for an arbitrary finite reflection group, in terms of distances
in the Bruhat graph.  On the other hand, the pattern avoidance criterion for permutations had earlier appeared in
independent work of Sj\"ostrand and of Gasharov and Reiner.  We give characterizations of the elements of the
hyperoctahedral groups satisfying Hultman's criterion that is in the spirit of those of Sj\"ostrand and of Gasharov and
Reiner.  We also give a pattern avoidance criterion using the notion of pattern avoidance defined by Billey and Postnikov.
\end{abstract}

\maketitle

\section{Introduction}

Let $w\in S_n$ be a permutation, and let $H_{i,j}$ denote the
hyperplane in $\mathbb{R}^n$ defined by $x_i=x_j$.  The {\bf inversion
  arrangement} of $w$ is the hyperplane arrangement in $\mathbb{R}^n$
given by $$\mathcal{A}_w:=\{H_{i,j}\mid i<j, w(i)>w(j)\}.$$

Any hyperplane arrangement $\mathcal{A}$ cuts $\mathbb{R}^n$ into a
number of {\bf chambers}, which are defined as the connected
components of $\mathbb{R}^n\setminus \mathcal{A}$.  Hence we can
associate an invariant $c(w)$, the number of chambers of
$\mathcal{A}_w$, to any permutation $w$.

On the other hand, if we let $[\id,w]$ denote the interval from the
identity to $w$ in Bruhat order, we can also associate to $w$ the
invariant $s(w):=\#[\id,w]$.  Following a conjecture of
Postnikov~\cite[Remark 24.2]{Pos}, Hultman, Linusson, Shareshian, and
Sj\"ostrand~\cite{HLSS} proved that $c(w)\leq s(w)$ for any permutation
$w$, and, furthermore, that $c(w)=s(w)$ if and only if $w$ avoids the
permutations 4231, 35142, 42513, and 351624.

The concepts defined above make sense for an arbitrary finite
reflection group $W$.  Given any element $w\in W$, one can define an
inversion arrangement $\mathcal{A}_w$ and an invariant $c(w)$ counting
its chambers.  There is also a notion of Bruhat order, so one can
associate to $w$ an invariant $s(w)$.  In slightly later work,
Hultman~\cite{Hult} observed that $c(w)\leq s(w)$ for any $w\in W$,
using essentially the same proof as in~\cite{HLSS}, and
gave a characterization of the elements $w$ for
which $c(w)=s(w)$ in terms of certain conditions on distances in the
Bruhat graph of $W$.  We will call the
elements $w\in W$ for which $c(w)=s(w)$ {\bf Hultman elements}.

The set of permutations avoiding 4231,
35142, 42513, and 351624 had previously appeared in two independent
but related places in the literature.  A permutation is said to be
{\bf defined by inclusions} if the interval $[\id,w]$ is defined by
inclusion conditions.  To be more precise, this means there are
(possibly empty) sets $B(w)$ and $T(w)$ such that $u\in[\id,w]$ if and
only if $\{u(1),\ldots,u(q)\}\subseteq \{1,\ldots,p\}$ for all $(p,q)\in
B(w)$ and $\{1,\ldots,p\}\subseteq \{u(1),\ldots,u(q)\}$ for all
$(p,q)\in T(w)$.  Gasharov and Reiner~\cite{GR} showed that a
permutation is defined by inclusions if and only if it avoids 4231,
35142, 42513, and 351624.  The {\bf right hull} of a permutation $w$
is the rectilinear convex hull of its graph and the points
$(1,1)$ and $(n,n)$.  A permutation $w$ satisfies the {\bf right hull
  condition} if $u\leq w$ for every permutation $u$ whose graph fits
in the right hull of $w$.  Sj\"ostrand~\cite{Sjo} showed that the same
set of permutations are the ones satisfying the right hull condition.

Our goal in this paper is to give a characterization of Hultman
elements for the hyperoctahedral groups $B_n$
that is in the spirit of the Gasharov--Reiner and Sj\"ostrand
characterizations for $S_n$, considering $B_n$ as a subgroup of $S_{2n}$.
We say that an element $w\in B_n\subseteq S_{2n}$ is {\bf defined
by pseudo-inclusions} if it is defined by inclusions (as an element of $S_{2n}$),
possibly with the additional condition $\#(\{u(1),\ldots,u(n)\}\cap \{1,\ldots,n\})=n-1$.
One can similarly define a relaxation of the right hull condition.
Furthermore, there is a generalization of pattern avoidance to arbitrary Coxeter
groups due to Billey and Postnikov~\cite{BP}, which we call {\bf BP avoidance}.
Our main theorem is as follows.  (Precise definitions are given in Section~\ref{sect:prelim}.)

\begin{theorem*}
Let $w\in B_n$.  Then the following are equivalent.
\begin{enumerate}
\item The number of chambers of the inversion arrangement $\mathcal{A}_w$ is equal to the number of elements in $[\id,w]$.
\item For any $u\leq w$, the directed distance from $u$ to $w$ in the Bruhat graph is the same as the undirected distance from $u$ to $w$.
\item The element $w$ is defined by pseudo-inclusions.
\item The element $w$ satisfies the relaxed right hull condition.
\item The element $w$ BP avoids $4231\in S_4$, $35142$, $42513\in S_5$, $351624\in S_6$, $563412$, $653421$, $645231$, $635241$, $624351$, $642531$, $536142$, $426153$, $462513$, $623451\in B_3$, $47618325$, $46718235$, $57163824$, $37581426$, $47163825$, $46172835$, $37518426$, $35718246$, $37145826$, $37154826$, $52618374$, $42681375$, $42618375$, $35172846\in B_4$, and $3517294a68$, $3517924a68$, $3617294a58\in B_5$.
\end{enumerate}
\end{theorem*}

The equivalence of the first 2 statements is due to Hultman.  Our proof showing the equivalence of the second condition with the remainder largely follows the proof of Hultman~\cite{Hult} recovering the pattern avoidance condition due to Hultman, Linusson, Shareshian and Sj\"ostrand~\cite{HLSS} from the distance condition, though a number of details are significantly
more complicated.

First we give a proof of the equivalence of the third and fourth conditions; this direct
proof of equivalence is new even for $S_n$.  Afterwards, given an element $w\in B_n$ satisfying the
relaxed right hull condition, we give a more complicated variant of Hultman's proof showing that the distance condition is satisfied in the Bruhat graph.

Second, given an element $w\in B_n$ not defined by pseudo-inclusions, we show that $w$ must BP contain
an element of $S_m$ or $B_m$ not defined by pseudo-inclusions for some $m\leq 5$.  A computer
calculation then shows that the above list is the minimal possible.  Another calculation shows that none
of the elements on the list satisfy the distance condition.  We finally prove that, if $w$ BP contains $u$ and
$u$ fails to satisfy the distance condtion, then $w$ must also fail to satisfy the distance condition.  It is possible
to avoid the first calculation by a more intricate but quite tedious version of the argument of Gasharov
and Reiner.

One can replace BP avoidance with avoidance of signed permutations at the cost of significantly
lengthening the list of patterns to be avoided.  We will see, however, that BP avoidance is the natural
notion to use in this context.

This work has several natural possible extensions.  First, it would be interesting to extend this
characterization to all finite reflection groups.  An extension to type D using the methods of this paper is likely possible and
may be the subject of a future paper.  An underlying principle guides the expected extension.  Fulton~\cite{Ful92} defined
the notion of the {\bf essential set} $E(w)$ of a permutation $w$, which is a set of conditions that characterizes when $u\in[\id,w]$.
Given $w\in S_n$, one can associate to each element of $(p,q)\in E(w)$ a permutation $v(p,q,r_w(p,q))$ such that $u$ fails the condition specified by $(p,q,r)$ if and only if $v\geq u(p,q,r)$.   Given an element $w\in B_n\subseteq S_{2n}$, we can similarly associate to each pair $$\{(p,q,r), (2n+2-p,2n-q, p-q-1+r)\}\subseteq E(w)$$ an element $$v(p,q,r)=v(2n+2-p, 2n-q, p-q-1+r)\in B_n$$ such that $u\in B_n$ fails both conditions if and only if $u\geq v(p,q,r)$.  It turns out that, for both $S_n$ and $B_n$, an element $w$ is Hultman if and only if $v(p,q,r)$ has only one reduced expression for all $(p,q,r)\in E(w)$.  Unfortunately, this condition cannot be stated in terms of the Coxeter-theoretic definition of coessential set given by Reiner, Yong, and the author in~\cite{RWY} because some conditions in $E(w)$ can be implied by other conditions in $E(w)$ when only considering elements of $B_n$.  We discuss this in Section~\ref{sect:essential} using recent work of Anderson~\cite{And}.  Given the nature of Hultman's proof and ours,
we expect the recent work of Gobet~\cite{Gobet} on analogues of cycle decompositions for finite reflection groups to be useful in any general proof of a general statement for all finite reflection groups.

Furthermore, several additional results
concerning the permutations defined by inclusions have potential analogues in $B_n$.
Lewis and Morales~\cite{LM} showed that, if $w$ is a permutation defined by inclusions, then
the number of invertible matrices over $\mathbb{F}_q$
supported on the complement of the diagram of $w$ is related to the
rank generating function of the interval $[\id,w]$ and hence is a $q$-analogue of $s(w)$,
but no analogous result for $B_n$ has been found.
Gasharov and Reiner~\cite{GR} gave a presentation of the cohomology ring
$H^*(X_w)$ for a Schubert variety corresponding to a permutation $w\in S_n$ defined
by inclusions; indeed they were originally interested only in the case where
$X_w$ is smooth and defined the notion of a permutation defined by
inclusions because their results naturally generalized to this case.  It would
be interesting to extend their results to Schubert varieties associated to
Hultman elements of $B_n$.  It is likely that any presentation of the
cohomology ring would involve the theta polynomials of Buch, Kresch, and
Tamvakis~\cite{BKT}.  Also, Ulfarsson and the author showed that the
Kostant--Kumar polynomials for permutations defined by inclusions are products
of not necessarily distinct roots~\cite[Cor. 6.6]{UW}.  It would be interesting to prove the
converse as well as extend this result to $B_n$.  Finally, Albert and
Brignall~\cite{AB} computed the generating function counting permutations defined
by inclusions; one possible path to an analogous result for $B_n$ would
be to understand the proof of Albert and Brignall in the context of the
staircase diagrams of Richmond and Slofstra~\cite{RS}.

We organize the paper as follows.  Section~\ref{sect:prelim} gives various preliminaries,
including some details of how BP pattern avoidance and cycle decomposition work for $B_n$.
The details on BP avoidance for $B_n$ do not seem to have previously appeared in print.
Section~\ref{sect:proof} contains the proof of our main theorem.  Section~\ref{sect:essential}
remarks on the relationship between our main theorem and the Coxeter-theoretic coessential
set.

\section{Preliminaries}
\label{sect:prelim}
\subsection{Bruhat order and inversion arrangements for permutations}

A {\bf transposition} is a permutation of the form $t_{i,j}:=(i\,\,j)$, swapping two elements and fixing all the others.
A {\bf simple transposition} is a transposition $s_i:=t_{i, i+1}$.  Let $w\in S_n$ be a permutation.  The {\bf length} of
$w$, denoted $\ell(w)$, is the minimum number of simple transpositions such that $w$ is a product of $\ell(w)$ simple
transpositions.  A pair $(i,j)$ with $1\leq i<j\leq n$ is an {\bf inversion} of $w$ if $w(i)>w(j)$.  It turns out that $\ell(w)$
is the number of inversions of $w$.  The {\bf absolute length} of $w$, denoted $\ell_T(w)$, is the minimum number of
(not necessarily simple) transpositions such that $w$ is a product of $\ell_T(w)$ transpositions.  Let $\cyc(w)$ be the
number of cycles in the cycle decomposition of $w$.  Then $\ell_T(w)=n-\cyc(w)$ for $w\in S_n$.

The symmetric group $S_n$ has a partial order known as {\bf Bruhat order}.  It can be defined as the transitive closure
of the covering relation where $u\prec w$ if $w=ut_{i,j}$ for some transposition $t_{i,j}$ and $\ell(w)=\ell(u)+1$.  Alternatively, Bruhat order can also be defined by the {\bf tableau criterion}.  Given $p,q$ with $1\leq p,q\leq n$, define $$r_w(p,q):=\#\{k\mid 1\leq k\leq q, p\leq w(k)\leq n\}.$$  The function $r_w$ is called the {\bf SW rank function} of $w$.  Then $u\leq w$ if and only if $r_u(p,q)\leq r_w(p,q)$ for all $p,q$.  We let $s(w)$ be the number of elements in the interval $[\id,w]$.

\begin{example}
Consider $w=35142$ and $u=13254$.  (All permutations in this paper are written in one line notation, except that $(i\,\,j)$ denotes the transposition switching $i$ and $j$.)  We have $u<w$ since $$u=13254\prec 31254\prec 32154\prec 35124\prec 35142=w.$$
On the other hand, the rank functions $r_u$ and $r_w$, displayed with values in a matrix, are
$$r_u=
\begin{matrix}
1 & 2 & 3 & 4 & 5 \\
0 & 1 & 2 & 3 & 4 \\
0 & 1 & 1 & 2 & 3 \\
0 & 0 & 0 & 1 & 2 \\
0 & 0 & 0 & 1 & 1
\end{matrix}\,\,\,
\mbox{ and }
r_w=
\begin{matrix}
1 & 2 & 3 & 4 & 5 \\
1 & 2 & 2 & 3 & 4 \\
1 & 2 & 2 & 3 & 3 \\
0 & 1 & 1 & 2 & 2 \\
0 & 1 & 1 & 1 & 1
\end{matrix}.$$
We see that every entry in $r_u$ is smaller than the corresponding entry of $r_w$.
\end{example}

The {\bf inversion arrangement} of $w$  is the hyperplane arrangement $\mathcal{A}_w$ in $\mathbb{R}^n$ consisting of the hyperplanes defined by $x_i-x_j=0$ for all inversions $(i,j)$ of $w$.  The {\bf chambers} of a hyperplane arrangement $\mathcal{A}$ are the connected components of $\mathbb{R}^n\setminus(\bigcup_{H\in\mathcal{A}} H)$.  We let $c(w)$ denote the number of chambers of $\mathcal{A}_w$.

\begin{example}
For $w=w_0$, where $w_0$ is the
permutation with $w_0(i):=n+1-i$ for all $i$, the arrangement $\mathcal{A}_{w_0}$ is the full braid arrangement for $S_n$, and $c(w)=n!$.  For $w=\id$, $\mathcal{A}_\id$ is the empty arrangement.  If $w=3412\in S_4$, then $\mathcal{A}_w$ consists of the hyperplanes defined by $x_4-x_1$, $x_3-x_2$, $x_3-x_1$, and $x_4-x_2$, and $c(3412)=14$.  If $w=4231\in S_4$, then $\mathcal{A}_w$ consists of the hyperplanes defined by $x_4-x_1$, $x_4-x_2$, $x_4-x_3$, $x_3-x_1$, and $x_2-x_1$, and $c(4231)=18$.
\end{example}

Given permutations $v\in S_m$ and $w\in S_n$ with $m\leq n$, we say that $w$ {\bf (pattern) contains} $v$ if there exist indices $1\leq i_1<\cdots<i_m\leq n$ such that, for all $j,k$ with $1\leq j<k\leq m$, $v(j)>v(k)$ if and only if $w(i_j)>w(i_k)$.  For example, $w=48631725$ contains $v=35142$ using indices $i_1=1$, $i_2=2$, $i_3=5$, $i_4=6$, and $i_5=7$.  We say $w$ {\bf (pattern) avoids}
$v$ if $w$ does not contain $v$.  For example, $w=68435271$ avoids $v=35142$.

Hultman, Linusson, Shareshian, and Sj\"ostrand~\cite{HLSS} proved the following.
\begin{theorem}
\label{thm:hlss}
Let $w\in S_n$ be a permutation.  Then
\begin{enumerate}
\item $c(w)\leq s(w)$.
\item $c(w)=s(w)$ if and only if $w$ avoids 4231, 35142, 42513, and 351624.
\end{enumerate}
\end{theorem}

\begin{example}
For $w=w_0$, the interval $[\id,w_0]$ is all of $S_n$, so $c(w)=s(w)=n!$.  For $w=3412$, we have $c(w)=s(w)=14$.  On the other hand, for $w=4231$, we have $c(w)=18$, but $s(w)=20$.
\end{example}

\subsection{Permutations defined by inclusions and the right hull property}

Given a permutation $w$, the {\bf diagram} for $w$, denoted $D(w)$, is defined as follows.  Draw a grid of $n\times n$ boxes, mark the entries of the permutation $w$ as in its permutation matrix, and cross out the boxes directly above and directly to the right of a permutation entry.  The remaining boxes are the diagram.  This means
$$D(w)=\{(p,q)\mid p>w(q), w^{-1}(p)>q\}.$$  Note that $(p,q)\in D(w)$ if and only if $(q,w^{-1}(p))$ is not an inversion.

The {\bf coessential set} $E(w)$ is the set of northeast-most boxes in a connected component of $D(w)$.  More precisely, $(p,q)\in E(w)$ if $(p,q)\in D(w)$, $(p-1,q)\not\in D(w)$, and $(p,q+1)\not\in D(w)$.  Alternatively,
$$E(w)=\{(p,q)\mid w^{-1}(p-1)\leq q< w^{-1}(p), w(q)<p\leq w(q+1)\}.$$
It is a lemma of Fulton~\cite{Ful92} that $u\leq w$ if $r_u(p,q)\leq r_w(p,q)$ for all $(p,q)\in E(w)$, and, furthermore, $E(w)$ is the unique minimal set that determines the Bruhat interval $[\id,w]$.  (To be precise, Fulton defined an {\bf essential set} using a NW rank function and proved a lemma about the essential set that is equivalent to the one stated here.)

\begin{example}
\label{ex:ess-819372564}
(This is copied from~\cite[Example 2.2]{UW}.)  Let $w=819372564$.  Then the diagram and coessential set of $w$ are as
in Figure~\ref{fig:ess-819372564}.  In particular, $$E(w)=\{(2,2),
(4,4), (4,6), (6,7), (9,2)\},$$ with $r_w(2,2)=1$, $r_w(4,4)=2$, $r_w(4,6)=3$, $r_w(6,7)=3$, and $r_w(9,2)=0$.

\begin{figure}[htbp]
\begin{tikzpicture}[scale=0.7]
\draw (0, 0) -- (9, 0) -- (9, 9) -- (0, 9) -- cycle;

\fill (0.5, 1.5) circle (4pt);
\draw[thick] (9, 1.5) -- (0.5, 1.5) -- (0.5, 9);
\fill (1.5, 8.5) circle (4pt);
\draw[thick] (9, 8.5) -- (1.5, 8.5) -- (1.5, 9);
\fill (2.5, 0.5) circle (4pt);
\draw[thick] (9, 0.5) -- (2.5, 0.5) -- (2.5, 9);
\fill (3.5, 6.5) circle (4pt);
\draw[thick] (9, 6.5) -- (3.5, 6.5) -- (3.5, 9);
\fill (4.5, 2.5) circle (4pt);
\draw[thick] (9, 2.5) -- (4.5, 2.5) -- (4.5, 9);
\fill (5.5, 7.5) circle (4pt);
\draw[thick] (9, 7.5) -- (5.5, 7.5) -- (5.5, 9);
\fill (6.5, 4.5) circle (4pt);
\draw[thick] (9, 4.5) -- (6.5, 4.5) -- (6.5, 9);
\fill (7.5, 3.5) circle (4pt);
\draw[thick] (9, 3.5) -- (7.5, 3.5) -- (7.5, 9);
\fill (8.5, 5.5) circle (4pt);
\draw[thick] (9, 5.5) -- (8.5, 5.5) -- (8.5, 9);

\draw (0, 1) -- (2, 1) -- (2, 0);
\draw (1, 1) -- (1, 0);
\draw (1, 2) -- (1, 8) -- (2, 8) -- (2, 2) -- cycle;
\draw (1, 3) -- (2, 3);
\draw (1, 4) -- (2, 4);
\draw (1, 5) -- (2, 5);
\draw (1, 6) -- (2, 6);
\draw (1, 7) -- (2, 7);
\draw (3, 2) -- (3, 6) -- (4, 6) -- (4, 2) -- cycle;
\draw (3, 3) -- (4, 3);
\draw (3, 4) -- (4, 4);
\draw (3, 5) -- (4, 5);
\draw (5, 3) -- (5, 6) -- (6, 6) -- (6, 4) -- (7, 4) -- (7, 3) -- cycle;
\draw (5, 4) -- (6, 4);
\draw (5, 5) -- (6, 5);
\draw (6, 4) -- (6, 3);

\node at (1.5, 0.5) {E};
\node at (1.5, 7.5) {E};
\node at (3.5, 5.5) {E};
\node at (5.5, 5.5) {E};
\node at (6.5, 3.5) {E};
\end{tikzpicture}
\caption{\label{fig:ess-819372564} Diagram and essential set for $w=819372564$.}
\end{figure}
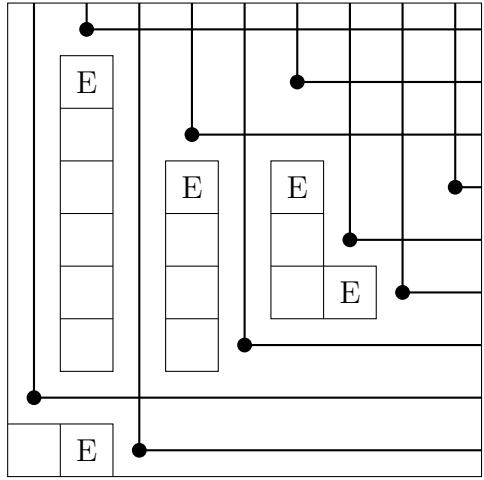
\end{example}

A permutation $w$ is {\bf defined by inclusions} if, for all $(p,q)\in E(w)$, we have $r_w(p,q)=\max(0,q-p+1)=r_\id(p,q)$.  Note that $r_w(p,q)\geq\max(0,q-p+1)$ for all $p,q$ for any permutation $w$, so being defined by inclusions means that the rank numbers are as small as possible for elements of the coessential set.  We use this terminology because $r_u(p,q)\leq 0$ if and only if $\{u(1),\ldots,u(q)\}\subseteq\{1,\ldots,p-1\}$ and $r_u(p,q)\leq q-p+1$ if and only if $\{1,\ldots, p-1\}\subseteq\{u(1),\ldots,u(q)\}$.  (The first statement is obvious from the definitions.  For the second, see Lemma~\ref{lem:NEedgeboxes}.)
Alternatively, $w$ is defined by inclusions if the Schubert variety $X_w$ is defined by conditions of the form $F_q\subseteq E_{p-1}$ (when $r_w(p,q)=0$) and the form $E_{p-1}\subseteq F_q$ (when $r_w(p,q)=q-p+1$).

Gasharov and Reiner~\cite{GR} showed that a permutation $w$ is defined by inclusions if and only if $w$ avoids 4231, 35142, 42513, and 351624.

\begin{example}
The permutation $w=819372564$ in Example~\ref{ex:ess-819372564}
is not defined by inclusions since $(4,4)\in E(w)$ and $r_w(4,4)=2\neq 4-4+1$.  We see that $w$ contains 4231 (in several ways).
\end{example}

Given a permutation $w$, define the {\bf right hull} of $w$ to be the set $H(w)$ of points $(i,j)$ satisfying {\it both} of the following:
\begin{itemize}
\item $1\leq i\leq w(k)$ for some $k\leq j$
\item $w(k)\leq i\leq n$ for some $k\geq j$.
\end{itemize}
We say that a permutation $v$ is in $H(w)$ and write $v\subseteq H(w)$ if $(v(j),j)\in H(w)$ for all $j$, $1\leq j\leq n$.  In other words, for this purpose,
we think of a permutation $v$ as the set of points $\{(v(j),j)\mid 1\leq j\leq n\}$.

Given any set $S\subseteq\mathbb{R}^2$ (which we index with matrix coordinates for the purposes of this paper), let the
{\bf rectilinear SW-NE hull} of $S$ be the set of points
$$H(S):=\{(p,q)\mid p_1\geq p\geq p_2, q_1\leq q\leq q_2 \text{ for some $(p_1, q_1)$ and $(p_2, q_2)$ in $S$}\}.$$
The right hull of $w$ can be thought of as the rectilinear SW-NE hull of $w$.

Note that, if $u\leq w$, then $u\subseteq H(w)$.  A permutation $w$ satisfies the {\bf right hull condition} if the converse is true,
meaning that, for all $u\subseteq H(w)$, we have $u\leq w$.

Sj\"ostrand~\cite{Sjo} showed that a permutation $w$ satisfies the right hull condition if and only if $w$ avoids 4231, 35142, 42513, and 351624.  He asked for an explanation of the connection between this result and the result of Gasharov and Reiner, which we
will provide in Section~\ref{sect:hull=dbi}.

\begin{example}
Let $w=819372564$ as in Example~\ref{ex:ess-819372564}.  The right hull $H(w)$ is the unshaded region in Figure~\ref{fig:rh-819372564}.  If $u=168523479$, then $u\in H(w)$, but $u\not\leq w$ since $r_u(4,4)=3$ but $r_w(4,4)=2$.  Hence $w$ does not satisfy the right hull condition.

\begin{figure}[htbp]
\begin{tikzpicture}[scale=0.7]
\draw (0, 0) -- (9, 0) -- (9, 9) -- (0, 9) -- cycle;

\fill (0.5, 1.5) circle (4pt);
\fill (1.5, 8.5) circle (4pt);
\fill (2.5, 0.5) circle (4pt);
\fill (3.5, 6.5) circle (4pt);
\fill (4.5, 2.5) circle (4pt);
\fill (5.5, 7.5) circle (4pt);
\fill (6.5, 4.5) circle (4pt);
\fill (7.5, 3.5) circle (4pt);
\fill (8.5, 5.5) circle (4pt);

\fill[color=gray] (0, 0) -- (0, 1) -- (2, 1) -- (2, 0) -- cycle;
\fill[color=gray] (2, 9) -- (9, 9) -- (9, 6) -- (6, 6) -- (6, 8) -- (2, 8) -- cycle;

\end{tikzpicture}

\caption{\label{fig:rh-819372564} Right hull for $w=819372564$.}
\end{figure}
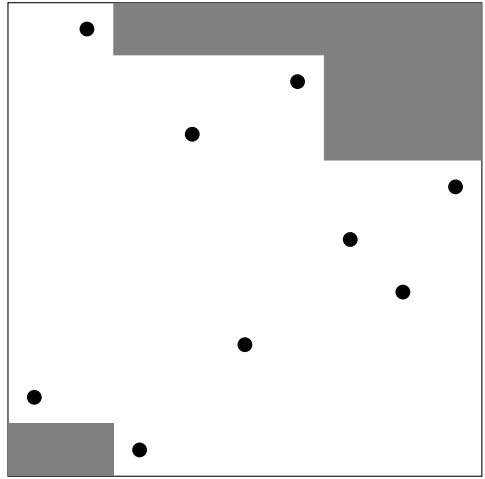
\end{example}

\subsection{Inversion arrangements for finite Coxeter groups}

Let $(W,S)$ be a {\bf Coxeter group}.  This is a group along with a distinguished set of generators $S$ such that the defining relations are $s^2=\id$ for all $s\in S$ and $(st)^{m(s,t)}=\id$ for all pairs $s,t\in S$, for some $m(s,t)\in\{2,3,\ldots\}\cup\{\infty\}$.  The symmetric group $S_n$ is a Coxeter group with generators $s_i=(i\,\,i+1)$.  In this case, we have $m(s_i,s_j)=2$ if $|j-i|\geq 2$ and $m(s_i,s_{i+1})=3$.  A {\bf Coxeter group isomorphism} $\phi: (W', S')\rightarrow (W, S)$ is a group isomorphism $\phi: W'\rightarrow W$ that induces a bijection between $S'$ and $S$.  For a general introduction to the combinatorics of Coxeter groups, see~\cite{BB}.

Given a Coxeter group $(W,S)$, a {\bf reflection} is a conjugate of an element of $S$; we denote the set of all reflections by $T$.  Given $w\in W$, the {\bf length} of $w$, denoted $\ell(w)$, is the minimum number of elements of $S$ whose product is $w$, and the {\bf absolute
length} of $w$, denoted $\ell_T(w)$, is the minimum number of elements of $T$ whose product is $w$.  {\bf Bruhat order} is the transitive closure of the covering relation defined by $v\prec w$ if $\ell(w)=\ell(v)+1$ and $w=vt$ for some $t\in T$.  As before, we let $s(w)$ denote the number of elements in the interval $[\id, w]$.  The {\bf inversion set} of $w$, denoted $\Inv(w)$, is the set
$$\Inv(w):=\{t\in T\mid wt<w\},$$ where $<$ denotes Bruhat order.  (Hultman~\cite{Hult} gives a definition of inversion equivalent to the condition that $tw<w$.  This forces him to multiply permutations backwards, which we wish to avoid.  None of the the results in this paper or in~\cite{Hult} see the difference between $w$ and $w^{-1}$, so no changes in the statements of results need to be made to account for this difference.)

A finite Coxeter group $(W,S)$ has a faithful action on $V(W)\cong\mathbb{R}^{|S|}$, known as the {\bf reflection representation}, in which elements of $S$ (and hence $T$) act as reflections.  Absolute length has the following interpretation due to Carter~\cite[Lemma 2]{Carter} in terms of the reflection representation.

\begin{lemma}
\label{lem:carter}
Let $w\in W$, and let $A\subseteq V(W)$ be the subspace consisting of all points fixed by $w$.  Then $\ell_T(w)=\codim A = |S|-\dim A$.
\end{lemma}

We will need the following simple corollary of this lemma.

\begin{corollary}
\label{cor:carter}
Let $w\in W$, with $w=t_1\cdots t_k$ for some $t_1,\ldots, t_k\in T$.  Suppose there exists a vector $\mathbf{v}$ and an index $i$ such
that $\mathbf{v}$ is fixed by $w$ but not by $t_i$.  Then $\ell_T(w)<k$.
\end{corollary}

\begin{proof}
We prove the contrapositive.  Suppose that $k=\ell_T(w)$, let $A$ be the subspace fixed by $w$, and let $H_{t_i}$ be the hyperplane fixed by $t_i$.  Note $A \supseteq \bigcap_{i=1}^k H_{t_i}$.  However, $\codim A=k$, so $A=\bigcap_{i=1}^k H_{t_i}$, which means that every vector fixed by $w$ must be fixed by every $t_i$.
\end{proof}

Given $w\in W$, define the inversion arrangement
$$\mathcal{A}_w:= \{ H_t\subseteq V(W)\mid t\in \Inv(w)\},$$ 
where
$$H_t:=\{\mathbf{v}\in V(W)\mid t\cdot\mathbf{v} = \mathbf{v}\}$$ is the hyperplane fixed by $t$.  As before, let $c(w)$ denote the number of chambers of $\mathcal{A}_w$.  Hultman~\cite{Hult} extended Theorem~\ref{thm:hlss} to an arbitrary finite Coxeter group using a condition on the {\bf Bruhat graph} $\mathcal{B}(W)$ of $W$.  This is the directed graph whose vertices correspond to the elements of $W$, with an edge from $u$ to $v$ if there exists $t\in T$ with $v=ut$ and $\ell(v)>\ell(u)$.  (Note that $v$ is not necessarily a cover of $u$ in Bruhat order, since the difference in length may be greater than 1.)

There are two possible notions of distance in $\mathcal{B}(W)$.  The {\bf directed distance} $\ell_D(u,w)$ is the length of a shortest directed path from $u$ to $w$ in $\mathcal{B}(W)$, while the {\bf undirected distance} $\ell_T(u,w)$ is the length of a shortest path from $u$ to $w$ ignoring the directions on the edges.  Note that $\ell_T(u,w)\leq\ell_D(u,w)$ by definition, and a theorem of Dyer~\cite{Dyer} states that $\ell_T(w)=\ell_T(\id,w)=\ell_D(\id,w)$ for all $w$.  Furthermore, $\ell_T(u,w)=\ell_T(w^{-1}u,\id)=\ell_T(w^{-1}u)$.  Hultman's theorem~\cite{Hult} is the following.

\begin{theorem}
\label{thm:hultman}
Let $(W,S)$ be a finite Coxeter group, and let $w\in W$.  Then
\begin{enumerate}
\item $c(w)\leq s(w)$.
\item $c(w)=s(w)$ if and only if, for all $u\leq w$, $\ell_D(u,w)=\ell_T(u,w)$.
\end{enumerate}
\end{theorem}

Hultman~\cite{Hult} goes on to give a much shorter and more conceptual proof of the difficult direction of Theorem~\ref{thm:hlss}(2) (that the right hull condition implies $c(w)=s(w)$) using Theorem~\ref{thm:hultman}(2).  Our goal is to find the analogues of the right hull condition and being defined by inclusions for $B_n$ as well as a pattern avoidance criterion for the elements where equality holds.  We will follow in outline the proof of Hultman.

\begin{example}
Let $w=4231\in S_4$, and let $u=1324$.  Note that $\ell_T(u,w)=2$, since both $u$ and $w$ are adjacent to the identity in $\mathcal{B}(S_4)$.  On the other hand, if $u<ut\leq w$, then $ut\in\{3124, 2314, 1423, 1342\}$, and no permutation in this set is adjacent to $w$ in $\mathcal{B}(S_4)$, so $\ell_D(u,w)>2$.
\end{example}

\subsection{Billey--Postnikov avoidance}

Billey and Postnikov~\cite{BP} introduced a generalization of pattern avoidance particularly well suited to algebraic combinatorics on Coxeter groups.  Let $(W,S)$ be a finite Coxeter group.  A {\bf parabolic subgroup} is a subgroup $P\subseteq W$ that consists of all
the elements fixing pointwise some subspace $A\subseteq V(W)$.  In other words, $P$ is parabolic if there exists $A\subseteq V(W)$, where $V(W)$ is the reflection representation, such that $$P=\{w\in W\mid w\cdot \mathbf{v}=\mathbf{v}\text{ for all $\mathbf{v}\in A$}\}.$$

Let $\mathcal{B}(P)$ be the subgraph of the Bruhat graph $\mathcal{B}(W)$ induced by the vertices in $P$.  Let $R\subseteq P$ be the elements corresponding
to vertices with exactly one incoming edge in $\mathcal{B}(P)$.  Then $(P, R)$ is a Coxeter group, and
we can define $\ell_P$ and $\leq_P$ as length and Bruhat order in $P$ (with respect to $R$).  Note that $\ell_P(v)\leq \ell_W(v)$
for all $v\in V$ and, if $v\leq_P v'$, then $v\leq_W v'$.  Since $v\leq_P v'$ implies $v\leq_W v'$, the Bruhat graph of $P$ (with respect to $R$) is $\mathcal{B}(P)$ as the edges are correctly directed.  Also, given $w\in P$, by Lemma~\ref{lem:carter}, $\ell_T(w)$ is the same whether we consider $w$ as an element of $P$ or of $W$ since,
when considering $w$ as an element of $P$, both $|S|$ and the codimension of the subspace fixed by $w$ are decreased by $\dim A$.  (The reader is cautioned that we have given a rather unusual definition of parabolic subgroups that has the drawback of applying only to the finite case.  It can be seen to be equivalent to the usual definition by, for example,~\cite[Section 5.2]{Kane}.  We have taken more care than usual in selecting $R$ as we need the positive roots of $P$ to be the positive roots of $W$ that lie in $A$.)  

\begin{example}
Consider $S_4$ acting on $\mathbb{C}^4$ by permuting the coordinates.  (Strictly speaking, the reflection representation of $S_4$ is the quotient of $\mathbb{C}^4$ by the line where $x_1=x_2=x_3=x_4$, but this is irrelevant for us.)  Consider the subspace $A$ consisting of points such that $x_1=x_2=x_4$.  This is fixed by the parabolic subgroup $$P=\{1234, 2134, 1432, 2431, 4132, 4231\}$$ of elements $w$ such that $w(3)=3$, and $R=\{2134, 1432\}$.
\end{example}

Billey and Postnikov~\cite{BP} define a {\bf flattening map} $\fl^W_P: W\rightarrow P$ as follows.  Given $w\in W$, let $\Inv(w)$ be
the set of inversions of $w$.  Then there is a unique element of $P$ whose inversions are $\Inv(w)\cap P$.  We let
$\fl^W_P(w)$ be this element.  Billey and Braden~\cite[Theorem 2]{BilBra} show that the flattening map satisfies the following.  (To be precise, they state the theorem with multiplication backwards from what we have written here.)

\begin{theorem}
\label{thm:flattening}
Let $W$ be a Coxeter group, $P\subseteq W$ a parabolic subgroup, and $\fl^W_P: W\rightarrow P$ the flattening map.  Then
\begin{enumerate}
\item The map $\fl^W_P$ is $P$-equivariant, meaning that $\fl^W_P(wv)=\fl^W_P(w)v$ for all $v\in P$, $w\in W$.
\item If $\fl^W_P(w)\leq_P \fl^W_P(wv)$ for some $w\in W$, $v\in P$, then $w\leq_W wv$.
\end{enumerate}
\end{theorem}

Now let $(W', S')$ and $(W,S)$ be arbitrary Coxeter groups, $v\in W'$, and $w\in W$.  We say that $w$ {\bf BP contains} $v$
if there exist a parabolic subgroup $P\subseteq W$ and an isomorphism $\phi: (W',S') \rightarrow (P,R)$ such that $\fl^W_P(w)=\phi(v)$.  Otherwise, $w$ {\bf BP avoids} $v$.

When $W=S_n$ and $W'=S_m$, the parabolic subgroups $P\subseteq W$ isomorphic to $W'$ can all be constructed by selecting some indices $1\leq i_1<\cdots<i_m\leq n$ and letting $P$ be the subgroup
$$P=\{w\in S_n\mid w(j)=j \mbox{ for all } j\not\in\{i_1,\ldots i_m\}\}.$$  In this case,
$$R=\{(i_j\,\, i_{j+1})\mid 1\leq j\leq m-1\}.$$

BP avoidance is almost the same as pattern avoidance for $W=S_n$.  If we take the isomorphism $\phi: (S_m, S) \rightarrow (P, R)$ given by $\phi(s_j)=(i_j\,\, i_{j+1})$, then $v:=\phi^{-1}(\fl^{S_m}_P(w))$ is the permutation such that $v(j)>v(k)$ if and only if $w(i_j)>w(i_k)$, so $w$ pattern contains $v$ as usual.  However, there is another Coxeter group isomorphism $\phi': (S_m, S) \rightarrow (P, R)$ given by $\phi'(s_j)=(i_{m-j}\,\, i_{m-j+1})$ (reversing the Dynkin diagram).  Since $\phi'(v)=\phi(w_0vw_0)$, where $w_0\in S_m$ is the element such that $w_0(i)=m+1-i$, we see that $w$ BP contains $v$ if and only if $w$ pattern contains either $v$ or $w_0vw_0$.

Note that the converse of Theorem~\ref{thm:flattening} holds for $W=S_n$ and $W'=S_m$, as can easily be seen from the tableau criterion.

\begin{proposition}
\label{prop:unflattenbruhat}
Let $w\in S_n$, $v\in S_m$, $Q\subseteq S_n$ the subgroup of permutations which fix every entry other than $i_1,\ldots, i_m$, and $\phi: S_m\rightarrow Q$ the order-preserving isomorphism.  If $w\leq_{S_n} w\phi(v)$, then $\fl^{S_n}_{Q}(w)\leq_{Q} \fl^{S_n}_{Q}(w)\phi(v)$.
\end{proposition}

\subsection{Conventions for type B}

We consider the group $B_n$ as the following subgroup of the symmetric group $S_{2n}$, as in the book of Bj\"orner and Brenti~\cite[Section 8.1]{BB}, so
$$B_n:=\{w\in S_{2n}\mid w(i)+w(2n+1-i)=2n+1\mbox{ for all } i, 1\leq i\leq n\}.$$
We can equivalently restate this condition by saying that $w\in B_n$ if $w_0ww_0=w$.

The group $B_n$ is a finite Coxeter group with simple generators $s_0:=(n\,\, n+1)$ and $s_i:=(n-i\,\,n-i+1)(n+i\,\,n+i+1)$ for all $i$ with $1\leq i\leq n-1$.  It acts on $\mathbb{R}^n$ with $s_0$ acting by reflection across the hyperplane $x_1=0$ and $s_i$ acting by reflection across the hyperplane $x_i-x_{i+1}=0$.

Bruhat order on $B_n$ turns out to be equal to the partial order induced from Bruhat order on $S_{2n}$ under this embedding.  (See, for example,~\cite[Cor. 8.1.9]{BB}.)  Given $w\in B_n$, we can define $E(w)$ by considering $w$ as an element of $S_{2n}$.  Because $w_0ww_0=w$, $E(w)$ has a rotational symmetry about $(n+1,n)$, so $(p,q)\in E(w)$ if and only if $(2n+2-p,2n-q)\in E(w)$.  Furthermore, $$r_w(2n+2-p,2n-q)=p-q-1+r_w(p,q),$$ which implies that, for $w\in B_n$, $r_w(p,q)=\max(0,q-p+1)$ if and only if $$r_w(2n+2-p,2n-q)=\max(0,p-q-1)=\max(0,(2n-q)-(2n+2-p)+1).$$

We need slight weakenings of the right hull condition and the concept of being defined by inclusions.  Given $w\in B_n$, we
say that $w$ is {\bf defined by pseudo-inclusions} if, for all $(p,q)\in E(w)$, we have either $r_w(p,q)=\max(0,q-p+1)=r_\id(p,q)$
(as in the definition of being defined by inclusions) or $p=n+1$, $q=n$, and $r_{n+1,n}(w)=1$.  Similarly, $w\in B_n$ satisfies the
{\bf relaxed right hull condition} if either, for every $u\in S_{2n}$ satisfying $u\subseteq H(w)$, we have $u\leq w$, or both
$r_w(n+1,n)=1$ and, for every $u\in S_{2n}$ satisfying $u\subseteq H(w)$ and $r_u(n+1,n)\leq 1$, we have $u\leq w$.
(Note that it is not sufficient to consider only $u\in B_n$, as Example~\ref{exa:426153} shows.  This is because the coessential set as we have
defined it here is an $S_{2n}$ concept and not truly appropriate for $B_n$, as discussed in Section~\ref{sect:essential}.)

\begin{example}
Let $w=362514\in B_3$.  Then $E(w)=\{(4,1), (4, 3), (4,5)\}$.  We have $r_w(4,1)=0$, $r_w(4,3)=1$, and $r_w(4,5)=2$.  The element $w$ is not defined by inclusions since $r_w(4,3)=1$, but it is defined by pseudo-inclusions.  Similarly, it does not satisfy the right hull condition but does satisfy the relaxed right hull condition.
\end{example}

The cycle decomposition of an element of $B_n$ has a special structure.  Since conjugation by $w_0$ fixes $w$ for any $w\in B_n$, conjugating any cycle $c$ in the cycle decomposition of $w$ also gives a cycle $\overline{c}=w_0cw_0$ of $w$.  If $c=\overline{c}$ then we let $\mathbf{c}=c=\overline{c}$ and call $\mathbf{c}$ an {\bf odd cycle} of $w$.  If $c\neq \overline{c}$, then we let $\mathbf{c}=c\overline{c}$ and call $\mathbf{c}$ an {\bf even cycle} of $w$.
Alternatively, $\mathbf{c}$ is an odd cycle if there is an odd number of $i$ such that $i\leq n$ and $\mathbf{c}(i)>n$, and $\mathbf{c}$ is an even cycle if there is an even number of $i$ such that $i\leq n$ and $\mathbf{c}(i)>n$.  It is easy to show that the reflection length of $w$ (as an element of $B_n$) is $n-\ecyc(w)$, where $\ecyc(w)$ is the number of even cycles of $w$.  (The cycle
decomposition for elements of $B_n$ and its relation to reflection length have appeared in the literature many times.  Our description of the cycle decomposition comes from~\cite{AAER} and our terminology from~\cite{ReinerSPS}.  The fact about reflection
length is stated in~\cite{CGG} and~\cite{Kal}.  Curiously,~\cite{CGG} uses the term ``balanced'' cycle to mean an even one,
but~\cite{Kal} uses the term ``balanced'' cycle to mean an odd one!)

Billey--Postnikov avoidance in type B works as follows.  There are two types of simple parabolic subgroups (by which we mean parabolic subgroups isomorphic to simple Coxeter groups) in $B_n$.  Given indices $1\leq i_1<i_2<\cdots<i_m\leq 2n$ with the property that $i_j+i_k\neq 2n+1$ for any $j,k$, we have a simple parabolic subgroup $P$ isomorphic to $S_m$ generated by
$(i_j\,\, i_{j+1})(2n+1-i_{j+1}\,\, 2n+1-i_j)$ for all $j$ with $1\leq j\leq m-1$.  If our isomorphism $\phi: S_m \rightarrow P$
is given by
$$\phi(s_j):=(i_j\,\, i_{j+1})(2n+1-i_{j+1}\,\, 2n+1-i_j),$$ then $\fl^{B_n}_P(w)$ is given by the relative order of $w(i_1),\ldots, w(i_m)$.
In particular, $\fl^{B_n}_P(w) = \fl^{S_{2n}}_Q(w)$, where $Q$ is the parabolic subgroup of $S_{2n}$ generated by the
elements $(i_j\,\, i_{j+1})$.  Hence, given $w\in B_n$, $w$ BP contains $v\in S_m$ if $w$ pattern contains $v$ using some indices $i_1<i_2<\cdots<i_m$ having the property that $i_j+i_k\neq 2n+1$ for any $j,k$.  (Note choosing indices $2n+1-i_m<\cdots<2n+1-i_1$ instead gets the same parabolic subgroup, but the isomorphism $\phi: S_m \rightarrow P$ is changed by
conjugation with $w_0$.)

\begin{example}
Let $w=52863174\in B_4$ and $v=4231\in S_4$.  Taking $i_1=1$, $i_2=2$, $i_3=5$, and $i_4=6$, we see that $w$ BP contains $v$.
\end{example}

Additionally, given indices $i_1<\cdots<i_m<i_{m+1}<\cdots<i_{2m}$ with $i_j+i_{2m+1-j}=2n+1$ for all $j$, we have a parabolic subgroup $P$ isomorphic to $B_m$ generated by $(i_m\,\, i_{m+1})$ and $(i_j\,\, i_{j+1})(i_{2m-j}\,\, i_{2m+1-j})$ for all $j$ with $1\leq j\leq m-1$.  Our isomorphism $\phi: B_m \rightarrow P$ must be given by $$\phi(s_0):=(i_m\,\, i_{m+1})$$ and $$\phi(s_j):=(i_{m-j}\,\,i_{m-j+1})(i_{m+j}\,\,i_{m+j+1}),$$ and $\fl^{B_n}_P(w)$ is given by the relative order of $w(i_1),\ldots, w(i_{2m})$.
Hence, $w\in B_n$ BP contains $v\in B_m$ if $w$ pattern contains $v$ using some indices $i_1<\cdots<i_m<i_{m+1}<\cdots<i_{2m}$ with $i_j+i_{2m+1-j}=2n+1$ for all $j$.

\begin{example}
Let $w=52863174\in B_4$ and $v=426351\in B_3$.  Taking $i_1=1$, $i_2=2$, $i_3=3$, $i_4=6$, $i_5=7$, and $i_6=8$, we see that $w$ BP contains $v$.
\end{example}

Note that, given $v\in S_m$, the results are different if we consider $v$ as an element of $S_m$ or $v$ as an element of $B_m$ via, for example, the standard embedding of $S_m$ in $B_m$.  Similarly, if $v\in S_m$ with $m$ even such that $v$ happens to be an element of $B_{m/2}$, we get different results considering $v$ as an element of $S_m$ or as an element of $B_{m/2}$.

\begin{example}
Let $w=52863174\in B_4$.  Then $w$ BP contains $v=4231\in S_4$, but $w$ does {\em not} BP contain $v=4231\in B_2$, nor does $w$ BP contain $v=42318675\in B_4$.
\end{example}

\begin{example}
This example shows that Proposition~\ref{prop:unflattenbruhat} is false for BP avoidance in $B_n$, even when restricted to simple parabolic subgroups and even when $w\in Q$.  Let $w=426153\in B_3$.  Consider the parabolic $Q$ generated by $R=\{r_1=132546, r_2=426153\}$, with the isomorphism $\phi: S_3\rightarrow Q$ given by $\phi(s_1)=r_1$ and $\phi(s_2)=r_2$.  (This is equivalent to taking $i_1=2$, $i_2=3$, and $i_3=6$.)  Then $w=r_2$, and $u=r_1=132546\leq_{B_3} w$, but $\phi^{-1}(u)=213 \not\leq_{S_3} \phi^{-1}(w)=132$.
\end{example}

\section{Proof of Main Theorem}
\label{sect:proof}

We now state our theorem and outline the proof, defering individual details to subsections.

\begin{theorem*}
Let $w\in B_n$.  Then the following are equivalent.
\begin{enumerate}
\item \label{cond:chambers}The number of chambers of the inversion arrangement $\mathcal{A}_w$ is equal to the number of elements in $[\id,w]$.
\item \label{cond:distance}For any $u\leq w$, the directed distance from $u$ to $w$ in the Bruhat graph is the same as the undirected distance from $u$ to $w$.
\item \label{cond:dbpi} The element $w$ is defined by pseudo-inclusions.
\item \label{cond:relaxedrh} The element $w$ satisfies the relaxed right hull condition.
\item \label{cond:patterns} The element $w$ BP avoids $4231\in S_4$, $35142$, $42513\in S_5$, $351624\in S_6$, $563412$, $653421$, $645231$, $635241$, $624351$, $642531$, $536142$, $426153$, $462513$, $623451\in B_3$, $47618325$, $46718235$, $57163824$, $37581426$, $47163825$, $46172835$, $37518426$, $35718246$, $37145826$, $37154826$, $52618374$, $42681375$, $42618375$, $35172846\in B_4$, and $3517294a68$, $3517924a68$, $3617294a58\in B_5$.
\end{enumerate}
\end{theorem*}

We use the notation $a=10$ in elements of $B_5$ to avoid confusion.  In what follows, we will call elements that satisfy Condition~\ref{cond:distance} {\bf Hultman elements}.

\begin{proof}
The Conditions~\ref{cond:chambers} and~\ref{cond:distance} are equivalent by Theorem~\ref{thm:hultman}.

Proposition~\ref{prop:hull=dbi} shows that Conditions~\ref{cond:dbpi} and~\ref{cond:relaxedrh} are equivalent.

Suppose Condition~\ref{cond:relaxedrh} holds, so $w\in B_n$ satisfies the relaxed right hull criterion.  Then, by Proposition~\ref{prop:rh->dist}, $w$ is Hultman.

Suppose Condition~\ref{cond:dbpi} does not hold, so $w\in B_n$ is not defined by pseudo-inclusions.  Then, by Proposition~\ref{prop:nodbpi->pattern}, $w$ BP contains some $v\in B_m$ or $S_{m+1}$, with $m\leq 5$, such that $v$ is not defined by pseudo-inclusions.  By computer calculation, $v$ BP contains one of the elements listed in Condition~\ref{cond:patterns}.  Since BP containment is transitive, Condition~\ref{cond:patterns} does not hold for $w$ either.

None of the patterns listed in Condition~\ref{cond:patterns} are Hultman, as shown by the data in Figure~\ref{fig:minnondbpi}.  Hence, by Proposition~\ref{prop:pattern->distance}, if $w\in B_n$ does not satisfy Condition~\ref{cond:patterns}, $w$ is not Hultman.
\end{proof}

\subsection{The right hull condition and being defined by inclusions}
\label{sect:hull=dbi}
In this section, we show that Sj\"ostrand's right hull condition and the Gasharov--Reiner condition
of being defined by inclusions are equivalent for permutations.
We first prove the following lemma, which is a restatement of~\cite[Lemma 3.2]{UW}.

\begin{lemma}
\label{lem:NEedgeboxes}
Let $w\in S_n$ and $1\leq p,q\leq n$.  Then the following are equivalent:
\begin{enumerate}
\item $r_w(p,q)=q-p+1$.
\item $w(k)\geq p$ for all $k>q$.
\item $\{1,\ldots,p-1\}\subseteq\{w(1),\ldots,w(q)\}$.
\end{enumerate}
\end{lemma}
\begin{proof}
Since $w$ is a bijection considered as a function $w:\{1,\ldots,n\}\rightarrow\{1,\ldots,n\}$, the second and third conditions are equivalent.  Now note that
\begin{align*}
n-p+1&=\#\{k\mid w(k)\geq p\}\\
&=\#\{k\leq q \mid w(k)\geq p\}+\#\{k> q\mid w(k)\geq p\}\\
&= r_w(p,q) + \#\{k> q\mid w(k)\geq p\},
\end{align*}
so $r_w(p,q)=q-p+1$ if and only if $\#\{k>q\mid w(k)\geq p\}=n-q$, and the last statement holds if and only if $w(k)\geq p$ for all $k>q$.
\end{proof}

To show the equivalence of the right hull condition and being defined by inclusions, we prove the following slightly stronger statement.

\begin{proposition}
\label{prop:hull=dbi}
Let $w\in S_n$.  Then $u\subseteq H(w)$ if and only if, for all $(p,q)\in E(w)$ with $r_w(p,q)=\max(0,q-p+1)$, $r_u(p,q)\leq r_w(p,q)$.  In particular, if $w$ satisfies the right hull condition if and only if $w$ is defined by inclusions.
\end{proposition}

\begin{proof}
First we assume that $u\not\subseteq H(w)$ and show that $r_u(p,q)>r_w(p,q)$ for some $(p,q)\in E(w)$ with $r_w(p,q)=\max(0,q-p+1)$.
If $u\not\subseteq H(w)$, either there exists some $i$ such that $w(j)<u(i)$ or $j>i$ for all $j$, or there exists some $i$ such that $w(j)>u(i)$ or $j<i$ for all $j$, or both.  In the first case, let $j_1$ be the largest $j$ such that $j_1\leq i$ and no $k$ satisfies both $k<j_1$ and $w(k)>w(j_1)$.  Note that $j_1$ exists since $j=1$ satisfies both conditions.  Furthermore, let $j_2$ be the smallest $j$ such that $j_2>i$ and no $k$ satisfies both $k<j_2$ and $w(k)>w(j_2)$.  If $w^{-1}(n)\leq i$, then neither $n<u(i)$ nor $w^{-1}(n)>i$ would hold, so $j=w^{-1}(n)$ satisfies both conditions, and $j_2$ exists.  Then $(w(j_1)+1,j_2-1)\in E(w)$, and $r_w(w(j_1)+1,j_2-1)=0$, but $r_u(w(j_1)+1, j_2-1)\neq 0$ since $i\leq j_2-1$ and $u(i)\geq w(j_1)+1$.

In the second case, let $j_1$ be the largest $j$ such that $j_1< i$ and no $k$ satisfies both $k>j_1$ and $w(k)<w(j_1)$.  Here $j_1$ exists since $j=w^{-1}(1)$ satisfies both conditions.  Also, let $j_2$ be the smallest $j$ such that $j_2\geq i$ and no $k$ satisfies both $k>j_2$ and $w(k)<w(j_2)$; $j=n$ satisfies both conditions, so $j_2$ exists. Then $(w(j_2),j_1)\in E(w)$.  Furthermore, there is no $k$ with $w(k)<w(j_2)$ and $k>j_1$, which, by Lemma~\ref{lem:NEedgeboxes}, implies $r_w(w(j_2),j_1)=j_1-w(j_2)+1$.  Also by Lemma~\ref{lem:NEedgeboxes}, $r_u(w(j_2),j_1)>r_w(w(j_2), j_1)$.

Now we assume that $u\subseteq H(w)$ and show that $r_u(p,q)=r_w(p,q)$ for all $(p,q)$ such that $r_w(p,q)=\max(0,q-p+1)$.  (We do not need to assume for this part that $(p,q)\in E(w)$, so in fact we prove a stronger statement.)  Suppose $(p,q)\in E(w)$ and $r_w(p,q)=0$.  Then $w(k)< p$ for all $q\leq j$.  Hence, if $u\subseteq H(w)$, $u(i)<p$ for all $i\leq q$ since, otherwise, there would exist $k'\leq k\leq q$ with $w(k')\geq u(k)\geq p$, contradicting our assumption that $r_w(p,q)=0$.  Therefore, $r_u(p,q)=0$.

Similarly, suppose $r_w(p,q)=q-p+1$.    By Lemma~\ref{lem:NEedgeboxes}, $w(k)\geq p$ for all $k>q$.  Hence, if $u\subseteq H(w)$, we also have $u(k)\geq p$ for all $k>q$ since, otherwise, there would exist $k'\geq k>q$ with $w(k')\leq u(k)<p$.  Therefore, $r_u(p,q)=q-p+1$.

To see the last statement of our proposition, note that $w$ fails to be defined by inclusions if and only if there exists $(p,q)\in E(w)$ such that $r_w(p,q)\neq\max(0,q-p+1)$, which means that, in order to have $u\leq w$, $u$ must satisfy a condition beyond $u\subseteq H(w)$.
\end{proof}

\subsection{Defined by pseudoinclusion elements are Hultman}

In this section, we prove the following proposition.

\begin{proposition}
\label{prop:rh->dist}
Suppose $w\in B_n$ satisfies the relaxed right hull condition.  Then $\ell_D(u,w)=\ell_T(u,w)$ for all $u\leq w$.
\end{proposition}

Our proof largely follows that of Hultman~\cite[Lemma 4.4]{Hult} showing that permutations satisfying the right hull condition are Hultman, but it is significantly more complicated because Proposition~\ref{prop:unflattenbruhat} does not hold for $B_n$ and arguments to work around this issue are necessary.

\begin{proof}
Let $w\in B_n$ satisfy the relaxed right hull condition.  We need to show that, given any $u\leq w$, we have $\ell_D(u,w)=\ell_T(u,w)$.  By induction on $\ell(w)-\ell(u)$, it suffices to show that there exists a reflection $t$
with $u<ut\leq w$ and $\ell_T(ut,w)=\ell_T(w^{-1}ut, \id)=\ell_T(w^{-1}u, \id)-1$.

We call a cycle {\bf trivial} if it consists simply of a pair of fixed points and {\bf nontrivial} otherwise.
First suppose that $w^{-1}u$ has no nontrivial even cycles and $u<w$.  Let $P$ be the parabolic subgroup generated by the reflections
$(i\,\, 2n+1-i)$ and $(i\,\, j)(2n+1-i\,\, 2n+1-j)$ where $w^{-1}u(i)\neq i$ and $w^{-1}u(j)\neq j$, or equivalently $u(i)\neq w(i)$ and $u(j)\neq w(j)$.  Observe that $uP=wP$.  Let $\tilde{u}=\fl^{B_n}_P(u)$ and $\tilde{w}=\fl^{B_n}_P(w)$.  Since, in this case,
$\fl^{B_n}_P(w)=\fl^{S_{2n}}_Q(w)$ where $Q$ is the parabolic subgroup of $S_{2n}$ generated by the reflections $(i\,\,j)$ for which $u(i)\neq w(i)$ and $u(j)\neq w(j)$, we have $\tilde{u}<\tilde{w}$ by Proposition~\ref{prop:unflattenbruhat} (and the fact that Bruhat order in $B_n$ is induced from Bruhat order for $S_{2n}$).  Hence, by the definition of Bruhat order, there exists a reflection $t\in P$ such that $\tilde{u}<_P\tilde{u}t\leq_P \tilde{w}$.  By Theorem~\ref{thm:flattening}, $u<ut\leq w$.  Since all nontrivial cycles of $w^{-1}u$ are odd, $w^{-1}u$ has maximal reflection length in $P$, so $\ell_T(w^{-1}ut, \id)=\ell_T(w^{-1}u, \id)-1$.

Now suppose there is a nontrivial even cycle $\mathbf{c}=c\overline{c}$ in the cycle decomposition of $w^{-1}u$.
We first show $wc<w$ in Bruhat order on $S_{2n}$.  (Note $wc\not\in B_n$.)  First note that $wc(i)=u(i)$ or $wc(i)=w(i)$ for all $i$.  Since $u<w$, $(u(i),i)\in H(w)$ for all $i$,
so $(wc(i),i)\in H(w)$ for all $i$.  Hence, if $w$ satisfies the (unrelaxed) right hull condition, $wc<w$.  Otherwise,
$r_w(n+1,n)=1$ and $r_u(n+1,n)\leq 1$.  If $r_u(n+1,n)=0$, then $r_{wc}(n+1,n)\leq 1$.  If $r_u(n+1,n)=1$, then 
let $a$ be the unique index such that $a\leq n$ and $w(a)>n$, with $\alpha:=w(a)$, and let $b$ be the unique index such
that $b\leq n$ and $w(b)>n$, with $\beta:=u(b)$.  Now, if $a=b$ or $\alpha=\beta$, then $r_{wc}(n+1,n)=1$, so $wc<w$.
Otherwise, if $c(b)=b$, then $(wc)(b)=w(b)\leq n$, so $r_{wc}(n+1,n)\leq 1$.

Finally, we are left with the case where $r_u(n+1,n)=r_w(n+1,n)=1$, $a\neq b$, $\alpha\neq\beta$, and $c(b)\neq b$.
Since $a\neq b$, $c(b)=w^{-1}(\beta)>n$.
We now show that $c(a)\neq a$.  Note that, as $b\leq n$, there must exist some $b'>n$ such that $c(b')\leq n$.  Furthermore, $b'\neq 2n+1-b$, since $\mathbf{c}$ is even, so $c(2n+1-b)=2n+1-b$.  Since $r_u(n+1,n)=1$ and $u\in B_n$, we can only have $i> n$ and $u(i)\leq n$ if $i=2n+1-b$, so $u(b')> n$.
It follows that, as $c(b')=w^{-1}(u(b'))\leq n$, $u(b')=\alpha$ and $c(b')=a$, so $c(a)\neq a$.
Hence, $wc(a)=u(a)$, and the only index $i\leq n$ with $wc(i)>n$ is $i=b$.
Therefore, $r_{wc}(n+1,n)=1$.  Since we also have $wc\in H(w)$, and $w$ satisfies
the relaxed right hull property, $wc<w$.

Now let $i_1<i_2\cdots<i_m$ be the indices such that $c(i_j)\neq i_j$.  (In particular, we let $m$ be the number of such indices.)  Consider the parabolic subgroup $P_c\subseteq B_n$ with the isomorphism $\phi: S_m \rightarrow P_c$ given by $\phi(s_j)=(i_j\,\, i_{j+1})(2n+1-i_{j+1}\,\,2n+1-i_j)$.  Note $\mathbf{c}\in P_c$ and, indeed, $\mathbf{c}$ is an element of maximal reflection length 
in $P_c$, so for any reflection $t\in P_c$, $\ell_T(w^{-1}ut,\id)=\ell_T(w^{-1}u,\id)-1$.

Moreover, $\fl^{B_n}_{P_c}(w\mathbf{c})=\fl^{S_{2n}}_Q(wc)$, and $\fl^{B_n}_{P_c}(w)=\fl^{S_{2n}}_Q(w)$, where $Q$ is the parabolic subgroup of $S_{2n}$ generated by the elements $(i_j\,\, i_{j+1})$.  Let $\tilde{u}=\fl^{B_n}_{P_c}(u)$ and
$\tilde{w}=\fl^{B_n}_{P_c}(w)$.  As $wc<w$, by Proposition~\ref{prop:unflattenbruhat}, $\fl^{S_{2n}}_Q(wc)<\fl^{S_{2n}}_Q(w)$, and
$$\tilde{u}=\fl^{B_n}_{P_c}(u)=\fl^{B_n}_{P_c}(w\mathbf{c})=\fl^{S_{2n}}_Q(wc)<\fl^{S_{2n}}_Q(w)=\fl^{B_n}_{P_c}(w)=\tilde{w}.$$
By the definition of Bruhat order, there exists $t\in P_c$ such that $\tilde{u}<_{P_c}\tilde{u}t\leq_{P_c} \tilde{w}$.  By
Proposition~\ref{thm:flattening}, $u<ut$.

We now show that $ut\leq w$.  For $i$ such that $\mathbf{c}(i)\neq i$, since $\tilde{u}t\leq \tilde{w}$, $$(i, ut(i))\in H(\{(w(i),i)\mid  \mathbf{c}(i)\neq i\})\subseteq H(w).$$  For $i$ such that $\mathbf{c}(i)=i$, we know $ut(i)=u(i)$, so $(ut(i),i)\in H(w)$ since $u\leq w$.  Hence, in the case where $w$ satisfies the (unrelaxed) right hull condition, $ut\leq w$.

Otherwise, $r_w(n+1,n)=1$, and $r_u(n+1,n)\leq 1$.
Let $t=(j\,\, k)(2n+1-j\,\, 2n+1-k)\in P_c$, and
assume without loss of generality that $j<k$ and that $c(j)\neq j$ while $\overline{c}(j)=j$.
Assume for contradiction that $r_{ut}(n+1,n)>r_u(n+1,n)$.  Then $j,u(j)\leq n$, and $k, u(k)>n$.
(We would also have $r_{ut}(n+1,n)>r_u(n+1,n)$ if $2n+1-k, u(2n+1-k)\leq n$ and $2n+1-j, u(2n+1-j)>n$,
but that amounts to the same condition.)
But then, if $j,u(j)\leq n$ and $k,u(k)>n$, we would have $k>n$ and $ut(k)\leq n$, while $j\leq n$ and $ut(j)\geq n$.
However, this implies $$\{(ut(j),j),(ut(k),k)\}\subseteq H(\{(w(i), i)\mid c(i)\neq i\})$$ since $\tilde{u}t\leq_{P_c}\tilde{w}$.
As before, let $a$ be the unique index such that $a\leq n$ and $w(a)>n$, which implies that $2n+1-a$ is the unique
index with $2n+1-a>n$ and $w(2n+1-a)\leq n$.  Since $\mathbf{c}$ is an even cycle,
either $c(a)=a$ or $c(2n+1-a)=2n+1-a$.  In the first case, there does not exist any $i$ with $c(i)\neq i$ such that
$i\leq n$ and $w(i)>n$, so $$(ut(j),j)\not\in H(\{(w(i), i)\mid c(i)\neq i\}),$$ a contradiction, and, in the second case,
$$(ut(k),k)\not\in H(\{(w(i), i)\mid c(i)\neq i\}).$$  Therefore, $r_{ut}(n+1,n)=r_u(n+1,n)\leq r_w(n+1,n)$.  Since
$w$ satisfies the relaxed right hull condition, $ut\leq w$.
\end{proof}

\subsection{Failure of inclusion conditions and pattern containment}

Now we show that an element that fails to be defined by pseudo-inclusions must fail to satisfy the distance condition.  We first show that any element not defined by pseudo-inclusions BP contains one of 27 elements.  We give a short proof here that relies on a computer calculation.  A longer proof by hand involving a tedious case-by-case analysis of the possible ways we can have $i_j+i_k=2n+1$ in the notation of the following proof is possible.

\begin{proposition}
\label{prop:nodbpi->pattern}
Suppose $w\in B_n$ is not defined by pseudo-inclusions.  Then $w$ BP contains some $v\in W$, where $W$ is either $S_{m+1}$ or
$B_m$ for some $m\leq 5$, such that $v$ is not defined by pseudo-inclusions.
\end{proposition}

\begin{proof}
If $w\in B_n$ is not defined by pseudo-inclusions, then either $(n+1,n)\in E(w)$ and $r_w(n+1,n)\geq 2$, or $(p,q)\in E(w)$ (with $(p,q)\neq (n+1,n)$) and $r_w(p,q)\geq \max(1, q-p+2)$.

Suppose $(n+1,n)\in E(w)$ and $r_w(n+1,n)\geq 2$.  Since $r_w(n+1,n)\geq 2$, there exist $i_1, i_2\leq n$ with $w(i_1),w(i_2)>n$.  Let $i_3=w^{-1}(n)$ and $i_4=n$.  Since $(n+1,n)\in E(w)$, $i_3, w(i_4)\leq n$.  Let $P$ be the parabolic
subgroup generated by the (not necessarily simple) reflections $(i_j\,\,i_k)(2n+1-i_k\,\,2n+1-i_j)$ and $(i_4\,\,2n+1-i_4)$. Depending on whether $i_3=i_4$ or not,  $P\cong B_3$ or $P\cong B_4$; let $m=3$ or $m=4$ respectively, and let $\phi: B_m\rightarrow P$ be the isomorphism.  Consider $v=\phi^{-1}(\fl^{B_n}_P(w))$.  Since $i_3<n$ and $w(i_3)=n$, we have $v^{-1}(m)\leq m$, and since $i_4=n$ with $w(i_4)<n$, we have $v(m)\leq m$.  This implies that $(m+1,m)\in E(v)$.  Furthermore,
$r_v(m+1,m)\geq 2$ since (the flattenings of) the points at $i_1$ and $i_2$ both count towards $r_v(m+1,m)$.

Now suppose $(p,q)\in E(w)$ with $(p,q)\neq (n+1,n)$ and $r_w(p,q)\geq \max(1, q-p+2)$.  Since $r_w(p,q)\geq \max(1, q-p+2)$, we must have some $i_1\leq q$ with $w(i_1)\geq p$ and some $i_2>q$ with $w(i_2)<p$.  Let
$i_3=w^{-1}(p-1)$, $i_6=w^{-1}(p)$, $i_4=q$, and $i_5=q+1$.  Since $(p,q)\in E(w)$, $i_3\leq i_4<i_5\leq i_6$, and $w(i_4)\leq w(i_3)<w(i_6)\leq w(i_5)$.

First we consider the case where $i_j+i_k\neq 2n+1$ for any $j,k$ with $1\leq j,k\leq 6$.  Let $P$ be the parabolic subgroup generated by the (not necessarily simple) reflections $(i_j\,\,i_k)(2n+1-i_k\,\,2n+1-i_j)$.  Depending on whether $i_3=i_4$ and $i_5=i_6$, $P\cong S_{m+1}$ for some $m$, $3\leq m\leq 5$, with some isomorphism $\phi: S_{m+1}\rightarrow P$.  Consider $v=\phi^{-1}(\fl^{B_n}_P(w))$.  Note $(p,q)$ will correspond to a box $(\tilde{p},\tilde{q})\in E(v)$, and $(w(i_1), i_1)$ and $(w(i_2),i_2)$ will force $r_v(\tilde{p},\tilde{q})\geq\max(1,\tilde{q}-\tilde{p}+2)$, so $v$ is not defined by inclusions.

If $i_j+i_k=2n+1$ for some $j,k$, then let $P$ be the parabolic subgroup generated by $(i_j\,\,i_k)(2n+1-i_k\,\,2n+1-i_j)$ and
$(i_j\,\,2n+1-i_j)$.  Depending on how many coincidences of indices we have, $P\cong B_m$ for some $m$, $3\leq m\leq 5$.  (We
cannot have $m=2$, because if $i_3=i_4=2n+1-i_5=2n+1-i_6$, then $(p,q)=(n+1,n)$.)  Again $(p,q)$ will correspond to a box $(\tilde{p},\tilde{q})\in E(v)$.  We cannot have $(\tilde{p},\tilde{q})=(m+1,m)$ because, in that case, we would have had $w(i_3)\leq n<w(i_6)$ and $i_4\leq n<i_5$, which would imply $(p,q)=(n+1,n)$.  Also, $(w(i_1),i_1)$ and $(w(i_2),i_2)$ will
force $r_v(\tilde{p},\tilde{q})\geq\max(1,\tilde{q}-\tilde{p}+2)$.

Hence, in all cases, $w$ BP contains some $v$ that is not defined by pseudo-inclusions.
\end{proof}

Using a computer program, we find all the elements of $S_{m+1}$ and $B_m$ with $m\leq 5$ that are not defined by (pseudo)-inclusions.  We then find among these elements the ones that do not BP contain some other element not defined by (pseudo)-inclusions.  These elements $w$ are listed in the table in Figure~\ref{fig:minnondbpi}.

\begin{figure}
\begin{tabular}{c|c|c|c|c}
$W$ & $w$ & $u$ & $\ell_D(u,w)$ & $\ell_T(u,w)$ \\
\hline
$S_4$ & 4231 & 2143 & 4 & 2 \\
$S_5$ & 35142 & 12435 & 5 & 3 \\
$S_5$ & 42513 & 13245 & 5 & 3 \\
$S_6$ & 351624 & 423156 & 6 & 4 \\
$S_6$ & 351624 & 126543 & 6 & 4 \\ 
$B_3$ & 426153 & 132546 & 4 & 2 \\
$B_3$ & 536142 & 142536 & 4 & 2 \\
$B_3$ & 563412 & 124356 & 4 & 2 \\
$B_3$ & 462513 & 135246 & 4 & 2 \\
$B_3$ & 635241 & 153426 & 4 & 2 \\
$B_3$ & 635241 & 241635 & 4 & 2 \\
$B_3$ & 642531 & 315264 & 4 & 2 \\
$B_3$ & 642531 & 153426 & 4 & 2 \\
$B_3$ & 645231 & 154326 & 4 & 2 \\ 
$B_3$ & 645231 & 351624 & 4 & 2 \\
$B_3$ & 623451 & 132546 & 4 & 2 \\
$B_3$ & 624351 & 135246 & 4 & 2 \\
$B_3$ & 624351 & 142536 & 4 & 2 \\
$B_3$ & 653421 & 214365 & 4 & 2 \\
$B_4$ & 35172846 & 12436578 & 5 & 3 \\
$B_4$ & 46172835 & 12536478 & 5 & 3 \\
$B_4$ & 57163824 & 14627358 & 5 & 3 \\
$B_4$ & 57163824 & 12654378 & 5 & 3 \\
$B_4$ & 47163825 & 13527468 & 5 & 3 \\
$B_4$ & 47163825 & 12645378 & 5 & 3 \\
$B_4$ & 52618374 & 14236758 & 5 & 3 \\
$B_4$ & 52618374 & 13254768 & 5 & 3 \\
$B_4$ & 47618325 & 13254768 & 5 & 3 \\
$B_4$ & 42681375 & 13427568 & 5 & 3 \\
$B_4$ & 42681375 & 13254768 & 5 & 3 \\
$B_4$ & 42618375 & 13245768 & 5 & 3 \\
$B_4$ & 37154826 & 12536478 & 5 & 3 \\
$B_4$ & 37154826 & 12463578 & 5 & 3 \\
$B_4$ & 37145826 & 12436578 & 5 & 3 \\
$B_4$ & 37581426 & 14627358 & 5 & 3 \\
$B_4$ & 37581426 & 12654378 & 5 & 3 \\
$B_4$ & 37518426 & 12645378 & 5 & 3 \\
$B_4$ & 37518426 & 14263758 & 5 & 3 \\
$B_4$ & 35718246 & 12463578 & 5 & 3 \\
$B_4$ & 46718235 & 12354678 & 5 & 3 \\
$B_5$ & 3617294a58 & 124365879a & 6 & 4 \\
$B_5$ & 3617294a58 & 125347869a & 6 & 4 \\
$B_5$ & 3517924a68 & 124365879a & 6 & 4 \\
$B_5$ & 3517924a68 & 124538679a & 6 & 4 \\
$B_5$ & 3517294a68 & 124356879a & 6 & 4 \\
\end{tabular}
\caption{\label{fig:minnondbpi}BP-containment-minimal non-Hultman elements in types A and B}
\end{figure}

For each $w$ in Figure~\ref{fig:minnondbpi}, we list all the elements $u<w$ such that $\ell_D(u,w)>\ell_T(u,w)$.  Note there is at least one
such $u$ for each $w$, so none of these elements $w$ are Hultman.

\subsection{Failure of containment and Hultman's condition}

To complete our proof, we show the following.

\begin{proposition}
\label{prop:pattern->distance}
Suppose $w\in W$ BP contains $v\in W'$ and $v$ is not Hultman.  Then $w$ is not Hultman.
\end{proposition}

\begin{proof}
By definition, we have some parabolic subgroup $P\subseteq W$ with an isomorphism $\phi: W'\rightarrow P$ such that $\fl^W_P(w)=\phi(v)$.  Since $v$ is not Hultman, there exists $u\leq_{W'} v$ such that $\ell_D(u,v)>\ell_T(u,v)$.  Now let
$x=w\phi(v^{-1}u)$.  Note that $$\ell_T(x,w)=\ell_T(w^{-1}x,\id)=\ell_T(w^{-1}x)=\ell_T(\phi(v^{-1}u))=\ell_T(v^{-1}u)=\ell_T(u,v).$$

Now consider a directed path $$x=x_0, x_1,\ldots, x_k=w$$ of length $k$ in $\mathcal{B}(W)$.  If $x_j\in wP$ for all $j$, then
$\fl^W_P(x_i)=\fl^W_P(w)(w^{-1}x_i)=\phi(v)w^{-1}x_i$ for all $i$, and we have a directed path $$u=\phi^{-1}(\fl^W_P(x_0)), \phi^{-1}(\fl^W_P(x_1)), \ldots, \phi^{-1}(\fl^W_P(x_k))=v$$ from $u$ to $v$ in $\mathcal{B}(W')$.  This is a path in $\mathcal{B}(W')$ because, for all $i$, $x_i=x_{i-1}t_i$ for some $t_i\in T\cap P$, so $v\phi^{-1}(w^{-1}x_i)=v\phi^{-1}(w^{-1}x_{i-1})\phi^{-1}(t_i)$.  This path is appropriately directed, meaning $v\phi^{-1}(w^{-1}x_i)>_{W'}v\phi^{-1}(w^{-1}x_{i-1})$ as, otherwise, we would have $\fl^W_P(x_i)<_P \fl^W_P(x_{i-1})$, and, hence by Theorem~\ref{thm:flattening}, since $x_i=x_{i-1}t_i$ and $t_i\in P$, we would have $x_i<x_{i-1}$, which is false by assumption.  Therefore, since $k$ is the length of a directed path from $u$ to $v$ in $\mathcal{B}(W')$, $k\geq\ell_D(u,v)>\ell_T(u,v)=\ell_T(x,w)$.

On the other hand, suppose there exists $x_j\not\in wP$.  For convenience, we choose the first such $j$, so $x_{j'}\in wP$ for all $j'<j$.  Let $t_i=x_{i}^{-1}x_{i-1}$ for all $i$, $1\leq i\leq k$.  Then $t_j\not\in P$.  However, $w^{-1}x=t_k\cdots t_1\in P$.  Hence, by Corollary~\ref{cor:carter}, $\ell_T(w^{-1}x)<k$.  Since any directed path from $x$ to $w$ in $\mathcal{B}(W)$ has length greater than $\ell_T(x,w)$, we must have that $w$ is not Hultman.
\end{proof}

\section{Coessential sets}
\label{sect:essential}
The purpose of this section is to describe the relation between our results and the Coxeter-theoretic coessential set defined in~\cite{RWY}.

Let $W$ be an arbitrary Coxeter group, and let $w\in W$.  The {\bf (Coxeter-theoretic) coessential set} of $w$, denoted $\mathcal{E}(w)$, is the set of minimal elements of $$\{v\in W\mid v\not\leq w\}.$$  An element of $W$ is {\bf basic} if it is an element of $\mathcal{E}(w)$ for some $w\in W$.  Hence, the set $\mathbf{B}(W)$ of all basic elements of $W$ is
$$\mathbf{B}(W):=\bigcup_{w\in W} \mathcal{E}(w).$$
An element $u\in W$ is {\bf bigrassmannian} if all reduced expressions for $u$ share the same initial and final element.  In other words, given $u=s_1\cdots s_{\ell(u)}=s'_1\cdots s'_{\ell(u)}$, then $s_1=s'_1$ and $s_{\ell(u)}=s'_{\ell(u)}$.
It is a theorem of Lascoux and Sch\"utzenberger~\cite{LS} (with a subsequent different proof due to Geck and Kim~\cite{GK}) that basic elements are bigrassmannian.  The converse is true for $W=S_n$ but not in general.

For $W=S_n$, elements of $\mathcal{E}(w)$ correspond to elements of $E(w)$ as follows.  Given a box at $(p,q)$ with $r=r_w(p,q)$, we have a unique minimal permutation $v=v(p,q,r)$ with $r_v(p,q)=r+1$.  To be precise, $$v(p,q,r)=1\cdots (q-r-1)p\cdots(p+r)(q-r)\cdots(p-1)(p+r+1)\cdots n,$$ written in 1-line notation.  The following proposition relates properties of the Coxeter-theoretic coessential set of an element $w\in S_n$ with the property of $w$ being defined by inclusions.

\begin{proposition}
Let $w\in S_n$.  An element $(p,q)\in E(w)$ corresponds to an element of $\mathcal{E}(w)$ with a unique reduced expression if and only if $r_w(p,q)=r_\id(p,q)=\max(0,q-p+1)$.
\end{proposition}

\begin{proof}
The element $$v(p,q,r)=1\cdots (q-r-1)p\cdots(p+r)(q-r)\cdots(p-1)(p+r+1)\cdots n$$ has a unique reduced expression precisely if $p=p+r$, in which case $r=0$ and $$v(p,q,r)=s_{p-1}\cdots s_{q-r}=s_{p-1}\cdots s_q$$ as a product of consecutive (in the Dynkin diagram) simple reflections, or if $p-1=q-r$, in which case $r=q-p+1$ and $$v(p,q,r)=s_{q-r}\cdots s_{p+r-1}=s_{p-1}\cdots s_q.$$  The two cases differ by whether the indices are increasing or decreasing.  (Note that, if $q-p+1>0$, so $q\geq p$, then $r=0$ is not possible.)
\end{proof}

This means that $w$ is defined by inclusions if and only if every element of $\mathcal{E}(w)$ has a unique reduced expression.

Given $W=B_n\subseteq S_{2n}$ and $w\in W$, the elements of $E(w)$ (considering $w$ as an element of $S_{2n}$) come in pairs that are rotationally symmetric around $(n+1,n)$.  In particular, if $(p,q)\in E(w)$, then $(2n+2-p, 2n-q)\in E(w)$, with $$r_w(2n+2-p, 2n-q)=q-p+1+r_w(p,q).$$

However, in some cases, the set $E(w)$ may not be minimal for the purposes of determining if $v\leq w$ for $v\in B_n$.  In particular, a pair $$\{(p,q), (2n+2-p,2n-q)\}\subseteq E(w)$$ may not be needed to determine if $v\leq w$.  For any $v\in B_n$, if $r_v(2n+2-p,q)\leq r$ with $p,q\leq n$, then, since either $v(i)>n\geq q$ or $v(2n+1-i)>n\geq q$ for any $i$, and in particular for $i$ with $p\leq i\leq n$ (so $n<2n+1-i<2n+2-p$), we must have $r_v(p,q)\leq r-(n-p+1)$.
Suppose for some $p,q\leq n$, both $(p,q)\in E(w)$ and $(2n+2-p,q)\in E(w)$, and $r_w(2n+2-p,q)=r_w(p,q)+p-q-1$.  Then, in this case, for any $v\in B_n$, if $r_v(2n+2-p,q)\leq r_w(2n+2-p,q)$, then automatically $r_v(p,q)\leq r_w(p,q)$.  Hence, to check
if $v\leq w$, it is not necessary to explicitly check if $r_v(p,q)\leq r_w(p,q)$.

Given $w\in B_n$, let $E'(w)=E(w)\setminus S$, where $S$ is the set of all redundant essential boxes.  To be precise, $S$ contains all boxes $(p,q), (2n+2-p, 2n-q) \in E(w)$ where $p,q<n$, $(2n+2-p,q)\in E(w)$, and $r_w(2n+2-p,q)=r_w(p,q)+p-n-1$.

In fact, these are the only redundant conditions.  Anderson~\cite{And} shows the following, in part using a geometric version of the above argument.

\begin{theorem}
The set $E'(w)$ is the unique minimal set satisfying both
\begin{itemize}
\item We have $(p,q)\in E'(w)$ if and only if $(2n+2-p,2n-q)\in E'(w)$.
\item For any $v\in B_n$, $v\leq w$ if and only if $r_v(p,q)\leq r_w(p,q)$ for all $(p,q)\in E'(w)$.
\end{itemize}
\end{theorem}

Furthermore, Anderson~\cite[p.13]{And} also gives for each pair of triples $\{(p,q,r), (2n+2-p, 2n-q,p-q-1+r)\}$ the minimal element $$v(p,q,r)=v(2n+2-p, 2n-q, p-q-1+r)\in B_n$$ such that $r_{v(p,q,r)}(p,q)>r$, thus explicitly giving a way of calculating $\mathcal{E}(w)$ as
$$\mathcal{E}(w)=\{v(p,q,r)\mid (p,q)\in E'(w), r=r_w(p,q)\}.$$
This correspondence is described in Table~\ref{table:essential-corresp} for the case where either $q<n$ or both $q=n$ and $p\geq n+1$.  Only the first half (meaning $w(1)\cdots w(n)$) of the elements $v(p,q,r)$ are listed, and $a\cdots b$ should be taken to be empty if $a>b$.  The remaining cases can be inferred from the symmetry $v(p,q,r)=v(2n+2-p, 2n-q, p-q-1+r)$.

\begin{figure}
\begin{tabular}{|c|l|}
\hline
 & $v(p,q,r)$ \\
\hline
$p+r\leq n$ & $1\cdots (q-r-1)p\cdots(p+r)(q-r)\cdots(p-1)(p+r+1)\cdots n$ \\
$p\leq n<p+r$ & $1\cdots(q-r-1)p\cdots n(2n+2-p)\cdots(n+r+1)(q-r)\cdots(n-r-1)$ \\
$n<p$, $p+q<2n+2$ & $1\cdots(q-r-1)p\cdots(p+r)(q-r)\cdots(2n-p-r)(2n+2-p)\cdots n$ \\
$p+q\geq 2n+2$ & $1\cdots(2n-p-r)(2n+2-p)\cdots q p\cdots(p+r)(q+1)\cdots n$ \\
\hline
\end{tabular}
\caption{$B_n$ elements corresponding to coessential boxes where $q<n$ or both $q=n$ and $p>n$.\label{table:essential-corresp}}
\end{figure}

Consulting the table and calculating reduced expressions gives the following proposition.

\begin{proposition}
Let $w\in B_n\subseteq S_{2n}$.  Then $w$ is defined by pseudo-inclusions if and only if, for all $(p,q)\in E(w)$, $v(p,q,r_w(p,q))$ has a unique reduced expression.
\end{proposition}

\begin{proof}
By the symmetry of $E(w)$, we only need to consider the case where $q<n$ or both $q=n$ and $p>n$.  If $q\geq p$, so $\max(0,q-p+1)=q-p+1$, then we have $q<n$ and hence $p+r\leq n$ by our assumptions, so
$$v(p,q,q-p+1)=1\cdots (p-2)p\cdots (q+1)(p-1)(q+2)\cdots n=s_{n-p+1}\cdots s_{n-q},$$
where the indices in the last expression are decreasing.

Otherwise, we first treat the case $r=0$.  If $p\leq n$, then
$$v(p,q,0)=1\cdots (q-1)pq\cdots(p-1)(p+1)\cdots n=s_{n-p+1}\cdots s_{n-q},$$
where the indices in the last expression are increasing.
If $p>n$, then either $p+q<2n+2$, in which case
$$v(p,q,0)=1\cdots(q-1)pq\cdots(2n-p)(2n+2-p)\cdots n=s_{p-n-1}\cdots s_0\cdots s_{n-q},$$
or $p+q\geq 2n+2$, in which case
$$v(p,q,0)=1\cdots(2n-p)(2n+2-p)\cdots qp(q+1)\cdots n=s_{p-n-1}\cdots s_0\cdots s_{n-q}.$$
In both cases, we take indices in the last expression to be first decreasing then increasing.

Finally, if $p=n+1$, $q=n$, and $r=1$, then
$$v(n+1,n,1)=1\cdots(n-2)(n+1)(n+2)=s_0 s_1 s_0.$$

These are the only elements of $B_n$ with a unique reduced expression, so the proposition is proved.
\end{proof}

Unfortunately, the proposition applies to $E(w)$ and not $E'(w)$, so it does {\em not} imply a statement about the Coxeter-theoretic coessential set $\mathcal{E}(w)$.  The following example makes this difference clear.

\begin{example}
\label{exa:426153}
Let $n=3$ and $w=426153$.  Then $E(w)=\{((3,2), (5,2), (5,4), (3,4)\}$.  We see that $w$ is not defined by pseudo-inclusions since $r_w(3,2)=1\neq\max(0, 2-3+1)$.  Moreover, $v(3,2,1)=v(5,4,1)=351624=s_1 s_0 s_2 s_1 = s_1 s_2 s_0 s_1$.  However, $E'(w)=\{(5,2), (3,4)\}$, with $r_w(5,2)=0$.  The permutation $v(5,2,0)=v(3,4,2)=153426=s_1 s_0 s_1$, which does have a unique reduced expression.  Note $153426<351624$ and $\mathcal{E}(w)=\{153426\}$.  The element $w$ fails the Hultman condition since, if $u=132546$, then $\ell_D(u,w)=4$ but $\ell_T(u,w)=2$.  It may be significant that this element also fails to satisfy Proposition~\ref{prop:unflattenbruhat}, as noted at the end of Section~\ref{sect:prelim}.
\end{example}

\section*{Acknowledgments}

I thank William Slofstra for helpful discussions.  This paper was completed while I am on sabbatical at the Department of Mathematics at the University of Illinois at Urbana--Champaign, and I thank the department for its hospitality.
I am supported by Simons Collaboration Grant 359792.

\end{document}